\documentclass[11pt,twoside,a4paper]{amsart}
\usepackage{amssymb}
\usepackage{amsmath,amscd}
\usepackage[initials,nobysame,shortalphabetic]{amsrefs}
\usepackage{verbatim}

\def\dbar{\bar\partial}
\def\R{{\mathbb R}}
\def\C{{\mathbb C}}

\def\F{{\mathfrak F}}

\def\Ok{{\mathcal O}}

\def\Re{{\rm Re\,  }}

\newcommand{\Com}[1]{}

\def\be{\begin{equation}}
\def\ee{\end{equation}}

\newtheorem{thm}{Theorem}[section]
\newtheorem{lma}[thm]{Lemma}
\newtheorem{cor}[thm]{Corollary}
\newtheorem{prop}[thm]{Proposition}

\theoremstyle{definition}

\newtheorem{df}{Definition}

\theoremstyle{remark}

\newtheorem{preremark}{Remark}
\newtheorem{preex}{Example}

\newenvironment{remark}{\begin{preremark}}{\end{preremark}}
\newenvironment{ex}{\begin{preex}}{\end{preex}}

\numberwithin{equation}{section}

\ifdefined\custompagenr
\fi

\begin{document}
\title[]{Log concavity for matrix-valued functions and a matrix-valued Pr\'ekopa theorem}

\author{Hossein Raufi}

\address{H. Raufi\\Department of Mathematics\\Chalmers University of Technology and the University of Gothenburg\\412 96 G\"OTEBORG\\SWEDEN}

\email{raufi@chalmers.se}

\begin{abstract}
We give two different definitions of what it means for a matrix-valued function, $g:\R^n\to\C^{r\times r}$, to be log concave, guided by similar notions in complex differential geometry. After discussing a few simple examples, we proceed to develop some of the basic properties associated with these new concepts. Finally, we prove a matrix-valued Pr\'ekopa theorem using a weighted, vector-valued Paley-Wiener theorem, and positivity properties of direct image bundles.
\end{abstract}

\maketitle

\section{Introduction}
\noindent In 1973 Pr{\'e}kopa, \cite{P}, proved the following celebrated theorem.

\begin{thm}
Let $\varphi:\R_t^m\times\R_y^n\to\R$ be convex and define $\tilde{\varphi}:\R^m\to\R$ through
$$e^{-\tilde{\varphi}(t)}=\int_{\R^n}e^{-\varphi(t,y)}dV(y).$$
Then $\tilde{\varphi}$ is convex.
\end{thm}

Motivated by recent results in complex analysis, in this article we prove a generalized version of this theorem in the setting where the functions involved are matrix valued. However, before we can formulate this result, we need to clarify what it means for a matrix-valued function $g:\R^n\to\C^{r\times r}$ to be log concave.

Let us, for simplicity, start with the scalar-valued case, and so we assume that $r=1$ and that $g$ is of the form
$$g(x_1,\ldots,x_n)=e^{-\varphi(x_1,\ldots,x_n)}$$
for some function $\varphi$. If $\varphi$, and hence $g$, is twice differentiable, then saying that $\varphi$ is convex is equivalent to saying that the Hessian of $\varphi$ is positive definite, i.e. for any $u\in\R^n$
$$0\leq\sum_{j,k=1}^nu_k\frac{\partial^2\varphi}{\partial x_k\partial x_j}u_j=-\sum_{j,k=1}^nu_k\frac{\partial^2\log g}{\partial x_k\partial x_j}u_j=-\sum_{j,k=1}^nu_k\frac{\partial}{\partial x_k}\left(g^{-1}\frac{\partial g}{\partial x_j}\right)u_j.$$
The point here is, of course, that the expression on the far right is well-defined even if $g$ is matrix-valued.

Now as already mentioned, our inspiration stems from related constructions in complex analysis. In this latter setting, if we regard $g$ as a function $g:\C^n_z\to\R$ which is independent of the imaginary part of $z$, $g=e^{-\varphi}$ can be thought of as a metric on a trivial line bundle, and then the log concavity of $g$ corresponds to requiring that this metric is positively curved. This prompts us to make the following definition.

\begin{df}\label{df:1}
We say that a matrix valued function $g:\R^n\to\C^{r\times r}$ is a \textit{metric} if:\\
(i) each matrix element of $g$ is twice differentiable,\\
(ii) $g(x)$ is a hermitian, strictly positive definite matrix for all $x\in\R^n$.\vspace{0.1cm}\\
Furthermore, for any metric $g$, we let
$$\Theta_{jk}^g:=\frac{\partial}{\partial x_k}\left(g^{-1}\frac{\partial g}{\partial x_j}\right),$$ 
where differentiation should be interpreted elementwise, juxtaposition denotes matrix-multiplication, and $g^{-1}$ denotes the matrix inverse of $g$, (which is well defined as $g$ is assumed to be strictly positively definite everywhere).
\end{df}
We have chosen $\C^{r\times r}$ instead of $\R^{r\times r}$, since this added generality comes for free and fits more nicely with the complex analytic correspondence. This is not crucial at all, and one can safely replace $\C^{r\times r}$ with $\R^{r\times r}$ throughout the paper.

If $g:\R^n\to\C^{r\times r}$ where $r\geq2$, then $g$ is the real-variable analogue of a metric on a vector bundle. In the complex setting, there exists two different, but equally important notions of curvature on vector bundles: being curved in the sense of Griffiths, and in the sense of Nakano. These motivate us to introduce the following corresponding notions in our real valued setting.

\begin{df}\label{df:2}
Let $g:\R^n\to\C^{r\times r}$ be a metric, and for $u,v\in\C^n$ set $(u,v)_g:=v^*gu$, where $v^*$ denotes the conjugate transpose of $v$.\\
(i) We say that $g$ is \textit{log concave in the sense of Griffiths} if for any vectors $u\in\C^r$ and $v\in\C^n$
$$\sum_{j,k=1}^n\big(\Theta_{jk}^gu,u\big)_gv_j\bar{v}_k\leq0.$$
(ii) We say that $g$ is \textit{log concave in the sense of Nakano} if for any $n$-tuple of vectors $\{u_j\}_{j=1}^n\subset\C^r$
$$\sum_{j,k=1}^n\big(\Theta_{jk}^gu_j,u_k\big)_g\leq0.$$
Strict log concavity, log convexity and strict log convexity are defined similarly.
\end{df}

It is clear that Nakano log concavity implies log concavity in the sense of Griffiths, (just choose $u_j=uv_j$), and that both these conditions reduce to the ordinary log concavity of functions when $r=1$. Griffiths and Nakano log concavity also coincide when $n=1$. In section 2 we will give some examples and develop other basic properties of metrics that are log concave in the sense of Griffiths and Nakano. In particular, we will discuss the correspondence between the real- and complex-variable theory in greater detail.

We are now ready to formulate our matrix-valued version of the Prekopa theorem.

\begin{thm}\label{thm:main}
Let $g:\R^n_y\times\R^m_t\to\C^{r\times r}$ be a metric. Assume that $g$ is log concave in the sense of Nakano and define $\tilde{g}:\R^n\to\C^{r\times r}$ through,
\be\label{eq:P}
\tilde{g}(t):=\int_{\R^n}g(y,t)dV(y).
\ee
Then $\tilde{g}$ is log concave in the sense of Nakano as well.
\end{thm} 
Here the integral should be interpreted elementwise, and we also assume that these integrals all converge.
 
Unfortunately, the definition of log concavity in the sense of Nakano, which is needed for Theorem \ref{thm:main}, is not overly intuitive. However, there exists a relatively simple class of metrics that are still 'genuinly' matrix-valued, (in constrast to e.g. diagonal matrices). Namely, in section 2, example \ref{ex:3}, we prove the following corollary.

\begin{cor}
Let $g:\R_y\times\R^m_t\to\C^{r\times r}$ be a metric of the form
$$g(y,t)=e^{-\varphi(y,t)}A(y),$$
where $\varphi$ is convex and $A$ is a metric on $\R$, such that $\Theta^A$ is negative definite. Then
$$\tilde{g}(t)=\int_{\R}e^{-\varphi(y,t)}A(y)dy,$$
is log concave in the sense of Nakano.
\end{cor}

Now the proof of Theorem \ref{thm:main} depends on two other theorems that are interesting in their own right, and to which we now turn. The first of these theorems states the following:

\begin{thm}\label{thm:PW}
Let $g:\R^n\to\C^{r\times r}$ be a metric, and define $\tilde{g}:\R^n\to\C^{r\times r}$ through,
\be\label{eq:PW}
\tilde{g}(\xi):=(2\pi)^n\int_{\R^n}e^{2\xi\cdot y}g(y)dV(y),
\ee
where we interpret the right hand side as elementwise integration, and the dot in the exponent represents the scalar product. Let furthermore
$$A^2(g):=\{F\in\Ok(\C^n;\C^r):\int_{\C^n}\|F(z)\|^2_{g(y)}dV(z)<\infty\}$$
where $\|F(z)\|^2_{g(y)}:=F(x+iy)^*g(y)F(x+iy)$, and let
$$L^2(\tilde{g}):=\{f\in L^2_{loc}(\R^n;\C^r):\int_{\R^n}\|f\|^2_{\tilde{g}}dV<\infty\}.$$
Then the following holds:\vspace{0.1cm}\\
(i) If $f\in L^2(\tilde{g})$, then the function
\be\label{eq:FT}
F(z):=\int_{\R^n}f(\xi)e^{-i\xi\cdot z}dV(\xi)
\ee
is in $A^2(g)$, (where once again, integration is assumed to be elementwise).\vspace{0.1cm}\\
(ii) Any $F\in A^2(g)$ can be written as in (\ref{eq:FT}) for some $f\in L^2(\tilde{g})$.\vspace{0.1cm}\\
(iii) For $f\in L^2(\tilde{g})$ and $F\in A^2(g)$ related as in (\ref{eq:FT}), we have that
\be\label{eq:parseval}
\int_{\C^n}\|F(z)\|^2_{g(y)}dV(z)=\int_{\R^n}\|f(\xi)\|^2_{\tilde{g}(\xi)}dV(\xi).
\ee
\end{thm}

When $r=1$, everything is scalar-valued, and the norms $\|F(z)\|^2_{g(y)}$ and $\|f(\xi)\|^2_{\tilde{g}(\xi)}$ correspond to $|F(z)|^2e^{-\phi(y)}$ and $|f(\xi)|^2e^{-\tilde{\phi}(\xi)}$, with the weights $\phi(y)=-\log g(y)$ and $\tilde{\phi}(\xi)=-\log\tilde{g}(\xi)$. In this setting, we see that Theorem \ref{thm:PW} is reduced to a weighted version of the Payley-Wiener theorem (see e.g. \cite{H},\cite{G} and \cite{SW}, Chapter III, Theorem 2.3). Thus Theorem \ref{thm:PW} can be seen as a generalization of this classic theorem to the weighted, vector-valued setting.

The second theorem is about the curvature of infinite rank, holomorphic vector bundles. Let $D=\Omega\times U$ be a domain in $\C^n_z\times\C^m_w$, where $\Omega$ is pseudoconvex, and let $h$ be a hermitian metric on $D$, i.e. $h:D\to\C^{r\times r}$ is smooth, and such that $h(z,w)$ is a hermitian and strictly positive-definite matrix for each $(z,w)\in\Omega\times\ U$. Then, for each fix $w\in U$, $h_w(\cdot):=h(\cdot,w)$ will be a hermitian metric on $\Omega$, and we let
\be\label{eq:fiber}
A^2_w(h):=\{F\in\Ok(\Omega;\C^r):\|F\|^2_{w,h}:=\int_{\Omega}\|F(z)\|^2_{h_w(z)}dV(z)<\infty\}.
\ee
We will assume that for any two $w_1,w_2\in U$, the norms $\|\cdot\|_{w_1,h}$ and $\|\cdot\|_{w_2,h}$ are equivalent. Then, for different $w\in U$, the (Bergman) spaces $A^2_w$ are then all equal as vector spaces, but their norms vary with $w$. Hence if we create an infinite rank vector bundle $E$, by setting $E_w:=A^2_w(h)$, we will get a trivial bundle equipped with a nontrivial metric.

The theorem now states the following:

\begin{thm}\label{thm:B}
If $h$ is positively curved in the sense of Nakano, then the hermitian vector bundle $(E,\|\cdot\|_{w,h})$ is positive in the sense of Nakano as well.
\end{thm}

Infinite rank vector bundles where the fibers are Bergman spaces, like $(E,\|\cdot\|_{w,h})$, have been extensively studied, (mainly by Berndtsson), in later years, (see e.g. \cite{B}, \cite{MT}, \cite{LY} and the references therein). In \cite{B}, (Theorem 1.1), it is shown that in the scalar-valued setting, (i.e. $r=1$ so that $h(z,w)=e^{-\varphi(w,z)}$), if $\varphi(w,z)=-\log h(w,z)$ is plurisubharmonic in $(z,w)$, then the bundle $(E,\|\cdot\|_{w,h})$ is positively curved in the sense of Nakano. Thus Theorem \ref{thm:B} is an extension of this result to the vector-valued setting. Closely related theorems, but with compact fibers, have previously been studied by Mourougane-Takayama \cite{MT}, and Liu-Yang \cite{LY}.

Now let us sketch how these two results can be combined to give a proof of Theorem \ref{thm:main}. The first observation is that although the relation (\ref{eq:PW}) at first might seem quite different from the one in Theorem \ref{thm:main}, the latter can actually be obtained rather easily from (\ref{eq:PW}). Namely, assume that the metric $g$ in Theorem \ref{thm:PW} depends on yet another variable $t\in\R^m$, i.e. $g:\R^n_y\times\R^m_t\to\C^{r\times r}$. For $t\in\R^m$ fix, (\ref{eq:PW}) then becomes
\be\label{eq:vb_pre}
\tilde{g}(\xi,t):=(2\pi)^n\int_{\R^n}e^{2\xi\cdot y}g(y,t)dV(y),
\ee
and choosing $\xi=0$ yields the sought for relation (\ref{eq:P}), (up to a constant).

However, if $g$, and hence $\tilde{g}$, depends on an extra variable $t$, then so will the Hilbert spaces $A^2(g)$ and $L^2(\tilde{g})$. This leads us to the study of (infinite rank) vector bundles, and so we can relate it to the setting of Theorem \ref{thm:B}.

In this latter setting we let $D=\C^n_z\times\C^m_w$ but at the same time we also note that for Theorem \ref{thm:PW}, it is of utmost importance that the metric on $A^2$ is independent of the real part of $z$. This will, however, certainly be the case if we start with a metric $g:\R^n_y\times\R^m_t\to\C^{r\times r}$, and think of it as a metric $h:\C^n_z\times\C^m_w\to\C^{r\times r}$ which is independent of the real parts of $z=x+iy$ and $w=s+it$, i.e. $h(z,w)=g(y,t)$.

Now we can define the metric $\tilde{h}:\R^n_\xi\times\C^m_w\to\C^{r\times r}$ through
\be\label{eq:herm_pre}
\tilde{h}(\xi,w)=(2\pi)^n\int_{\R^n}e^{2\xi\cdot y}h(y,w)dV(y)
\ee
and let
\be\label{eq:L2}
L^2_w(\tilde{h}):=\{f\in L^2_{loc}(\R^n;\C^r):\|f\|^2_{w,\tilde{h}}:=\int_{\R^n}\|f(\xi)\|^2_{\tilde{h}_w(\xi)}dV(\xi)<\infty\}.
\ee
Similar to the construction of $E$ we see that for different $w\in\C^m$, these $L^2$-spaces are equal as vector spaces, but with norms that vary with $w$. Hence in this way we can construct a second infinite rank, trivial, hermitian vector bundle $(\tilde{E},\|\cdot\|_{w,\tilde{h}})$ over $\C^m$, equipped with a nontrivial metric, by setting $\tilde{E}_w:=L^2_w(\tilde{h})$.

What Theorem \ref{thm:PW} says is basically that in this setting, the two vector bundles $(E,\|\cdot\|_{w,h})$and $(\tilde{E},\|\cdot\|_{w,\tilde{h}})$ are isometrically isomorphic. Combined with Theorem \ref{thm:B} we hence get that if $h$ is Nakano positive, (which will be the case if $g$ is Nakano log concave), then the bundle $(\tilde{E},\|\cdot\|_{w,\tilde{h}})$ will be Nakano positive as well. Thus for the proof of Theorem \ref{thm:main} we need to check that this implies the Nakano log concavity of $\tilde{g}$.

We end this introduction with a few words about the proofs of Theorem \ref{thm:PW}-\ref{thm:B}. We learned about the scalar-valued case of Theorem \ref{thm:PW} from the master's thesis of Jakob Hultgren, performed at the mathematics department of Chalmers University of Technology and Gothenburg University. However, since this thesis was never properly published, we have included a full proof in section 3, where we basically just adapt the scalar-valued proof to the vector-valued setting.

Of course, the master's thesis was not the first time these type of results were obtained. We have definitely not performed any exhaustive reference study, but refer once again to \cite{H}, \cite{G}, and \cite{SW}, and the references therein.

Finally, our proof of Theorem \ref{thm:B} use the same ideas as in the scalar-valued case, (\cite{B}, Theorem 1.1). These are the Griffiths subbundle formula, and H\"ormander $L^2$-estimates for the $\dbar$-equation. This is parallel to the proof of Brascamp and Lieb of the Pr\'ekopa theorem, \cite{BL}. Since these $L^2$-estimates in the vector-valued setting are somewhat different from the estimates in the scalar-valued setting, the second part of the proof is a little different from the proof in \cite{B}.

\section*{Acknowledgments}
\noindent It is a pleasure to thank Bo Berndtsson and Jakob Hultgren for inspiring and fruitful discussions.

\section{Examples and basic properties of log concave metrics}
\noindent As mentioned in the introduction, the log concavity properties of Definition \ref{df:2} are well-known concepts in complex differential and algebraic geometry. The aim of this section is to investigate which of the basic complex analytic properties that can be given an appropriate interpretation in the real setting. Of course, these properties can always be shown to hold in the 'complex sense', namely by interpreting the metrics $g:\R^n\to\C^{r\times r}$ as metrics $h:\C^n_z\to\C^{r\times r}$ that do not depend on the real part of $z$, but we have tried to come up with real proofs as far as possible.

We begin by looking at a few simple examples.  

\begin{ex}\label{ex:1}
Assume that the metric $g:\R^n\to\R^{r\times r}$ is diagonal, $g(x)=\textrm{diag}\{g_1(x),\ldots,g_r(x)\}$. Any reasonable definition of what it means for $g$ to be log concave, should in this case be that $g_j$ is log concave for all $j=1,\ldots,r$. Let us check that this is so for log concavity in both the Griffiths and Nakano sense.

It is immediate that
$$\Theta_{jk}=\textrm{diag}\left\{\frac{\partial^2\log g_1}{\partial x_j\partial x_k},\ldots,\frac{\partial^2\log g_r}{\partial x_j\partial x_k}\right\},$$
and so
$$g\cdot\Theta_{jk}=\textrm{diag}\left\{g_1\frac{\partial^2\log g_1}{\partial x_j\partial x_k},\ldots,g_r\frac{\partial^2\log g_r}{\partial x_j\partial x_k}\right\}.$$
Hence, if we let $u_j=(u_j^1,\ldots,u_j^r)$ we get that
{\setlength\arraycolsep{2pt}
\begin{eqnarray}
&&\sum_{j,k=1}^n\big(\Theta_{jk}u_j,u_k\big)_g=\sum_{j,k=1}^nu_k^*g\cdot\Theta_{jk}u_j=\nonumber\\
&&\qquad=g_1\sum_{j,k=1}^nu_j^1\frac{\partial^2\log g_1}{\partial x_j\partial x_k}\bar{u}_k^1+\ldots+g_r\sum_{j,k=1}^nu_j^r\frac{\partial^2\log g_r}{\partial x_j\partial x_k}\bar{u}_k^r.\nonumber
\end{eqnarray}}
\!\!On the other hand, in the Griffiths case we have
{\setlength\arraycolsep{2pt}
\begin{eqnarray}
&&\sum_{j,k=1}^n\big(\Theta_{jk}u,u\big)_gv_j\bar{v}_k=\nonumber\\
&&\qquad=g_1|u^1|^2\sum_{j,k=1}^nv_j\frac{\partial^2\log g_1}{\partial x_j\partial x_k}\bar{v}_k+\ldots+g_r|u^r|^2\sum_{j,k=1}^nv_j\frac{\partial^2\log g_r}{\partial x_j\partial x_k}\bar{v}_k.\nonumber
\end{eqnarray}}
\!\!Since $g$ is assumed to be strictly positive definite, each diagonal element $g_j$ is a strictly positive function, so the factors in front of the sums do not change the signs. Thus we see that, as expected, the definitions of log concavity in the sense of Griffiths and Nakano both in this case are equivalent to requiring that $g_j$ is log concave for each diagonal element of $g$.

In this setting, Theorem \ref{thm:main} says that if this is the case, then
$$\tilde{g}_j(t)=\int_{\R^n}g_j(y,t)dV(y)$$
is log concave on $\R^m$, for all $j=1,\ldots,r$. Clearly, this is just Prekopa's theorem.
\end{ex}

\begin{ex}\label{ex:2}
For a slightly (but only \textit{slightly}!) more 'genuin' matrix-valued example, let us study metrics of the form
$$g(x)=f(x)\cdot C,$$
where $f\in C^2(\R^n)$ is a strictly positive function and $C\in\C^{l\times l}$ is a constant, hermitian, strictly positive definite matrix.

In this case we have that
$$\Theta_{jk}=\frac{\partial}{\partial x_k}\left(\frac{1}{f}C^{-1}\frac{\partial f}{\partial x_j}C\right)=\frac{\partial^2\log f}{\partial x_j\partial x_k}\cdot I,$$
where $I$ denotes the identity matrix.

The Griffiths log concavity condition hence becomes,
$$\sum_{j,k=1}^n\big(\Theta_{jk}u,u\big)_gv_j\bar{v}_k=\|u\|^2_g\sum_{j,k=1}^nv_j\frac{\partial^2\log f}{\partial x_j\partial x_k}\bar{v}_k\leq0,$$
for all $u\in\C^r$ and $v\in\C^n$, while the Nakano log concavity condition becomes,
\be\label{eq:N_ex}
\sum_{j,k=1}^n\big(\Theta_{jk}u_j,u_k\big)_g=\sum_{j,k=1}^n\frac{\partial^2\log f}{\partial x_j\partial x_k}\big(u_j,u_k\big)_g\leq0,
\ee
for all $n$-tuples $\{u_j\}_{j=1}^n\subset\C^r$.

It is clear that the Griffiths log concavity condition is equivalent to the log concavity of $f$, but at first sight, it looks like the same thing does not necessarily hold in the Nakano case. However $f$ log concave does in fact also imply that $g$ is Nakano log concave. (For the converse property, just choose $u_j=uv_j$ in (\ref{eq:N_ex}).)

This follows from a theorem due to Schur, (\cite{L} Theorem 7, Chapter 10), which states that if $A=(a_{jk})$ and $B=(b_{jk})$
are positive definite matrices, and we define $M=(m_{jk})$ as the elementwise product of $A$ and $B$,
$$m_{jk}=a_{jk}b_{jk},$$
then $M$ is also a positive definite matrix.

Now let
$$a_{jk}=-\frac{\partial^2\log f}{\partial x_j\partial x_k}\quad\textrm{and}\quad b_{jk}=(u_j,u_k)_g.$$
Then the log concavity of $f$ implies that $A$ is positive definite, and $B$ is also positive definite since it is the Gram matrix of the $n$-tuple $\{u_j\}_{j=1}^n$. Hence their elementwise product, $M$, is also positive definite, and if we let $\mathbf{1}\in\R^r$ denote a vector of ones, we have that
$$\sum_{j,k=1}^n\big(\Theta_{jk}u_j,u_k\big)_g=-(M\mathbf{1},\mathbf{1})_I\leq0,$$
so $g$ is log concave in the sense of Nakano.
\end{ex}

\begin{ex}\label{ex:3}
Finally, we have the class of metrics described in the introduction. Namely, metrics $g:\R_y\times\R^m_t\to\C^{r\times r}$ of the form, $g(y,t)=e^{-\varphi(y,t)}A(y)$, where $\varphi$ is convex, and $A$ is a metric on $\R$ such that,
$$\Theta^A=\frac{d}{dy}\left(A^{-1}\frac{dA}{dy}\right),$$
is negative definite.

For any $(m+1)$-tuple $\{u_j\}_{j=1}^{m+1}\subset\C^r$ we then have that
$$\sum_{j,k=1}^{m+1}\big(\Theta_{jk}u_j,u_k\big)_g=-\sum_{j,k=1}^{m+1}\frac{\partial^2\varphi}{\partial t_j\partial t_k}\big(u_j,u_k\big)_g+\big(\Theta^Au_{m+1},u_{m+1}\big)_g.$$
From the previous example we know that the convexity of $\varphi$ implies that the sum on the right hand side is non-positive. Together with the negative definiteness of $\Theta^A$, we get that $g$ is log concave in the sense of Nakano.
\end{ex}

Now let us turn to the general study of log concave metrics. The following basic properties are the real variable analogues of corresponding results in complex differential geometry.

\begin{prop}\label{prop:1}
Let $g:\R^n\to\C^{r\times r}$ be a metric. Then the following holds:\vspace{0.2cm}\\
(i) For any pair of vectors $u,v\in\C^r$,
$$(\Theta_{jk}u,v)_g=(u,\Theta_{kj}v)_g\qquad\textrm{for all }j,k=1,\ldots,n.$$
(ii) In the Griffiths setting, if $g$ is log concave, then $g^{-1}$ is log convex, but the corresponding result does not hold in general in the Nakano case.\vspace{0.2cm}\\
(iii) If $g$ is log concave in the sense of Griffiths, then $\det g$ is a log concave function.\vspace{0.2cm}\\
(iv) Assume that $g$ is log concave either in the sense Griffiths (or Nakano), and let $f:\R^n\to\R$ be a log concave function. Then the metric $\tilde{g}:=f\cdot g$ will also be log concave in the sense of Griffiths (or Nakano).
\end{prop}

\begin{proof}
By differentiating the equality
$$I=g\cdot g^{-1}$$
with respect to $x_j$, we see that,
$$\frac{\partial g^{-1}}{\partial x_j}=-g^{-1}\frac{\partial g}{\partial x_j}g^{-1}.$$
Using this in the definition of $\Theta_{jk}^{g^{-1}}$ yields
$$\Theta_{jk}^{g^{-1}}:=\frac{\partial}{\partial x_k}\left(g\frac{\partial g^{-1}}{\partial x_j}\right)=-\frac{\partial}{\partial x_k}\left(\frac{\partial g}{\partial x_j}g^{-1}\right).$$
Now since $g$ is hermitian, so is $g^{-1}$ and all the partial derivatives of $g$, which in turn yields that,
\be\label{eq:1}
\big(\Theta_{jk}^{g}\big)^*=\frac{\partial}{\partial x_k}\left(\frac{\partial g}{\partial x_j}g^{-1}\right)=-\Theta_{jk}^{g^{-1}}.
\ee

On the other hand, by direct computation we get that
$$\Theta_{jk}^{g}=\frac{\partial}{\partial x_k}\left(g^{-1}\frac{\partial g}{\partial x_j}\right)=-g^{-1}\frac{\partial g}{\partial x_k}g^{-1}\frac{\partial g}{\partial x_j}+g^{-1}\frac{\partial^2 g}{\partial x_j\partial x_k},$$
and
$$\Theta_{jk}^{g^{-1}}=-\frac{\partial}{\partial x_k}\left(\frac{\partial g}{\partial x_j}g^{-1}\right)=-\frac{\partial^2 g}{\partial x_j\partial x_k}g^{-1}+\frac{\partial g}{\partial x_j}g^{-1}\frac{\partial g}{\partial x_k}g^{-1}.$$
Hence
$$-g^{-1}\Theta_{kj}^{g^{-1}}=\Theta_{jk}^{g}g^{-1},$$
or equivalently
\be\label{eq:2}
\Theta_{jk}^{g}=-g^{-1}\Theta_{kj}^{g^{-1}}g.
\ee

Altogether, (\ref{eq:1}) and (\ref{eq:2}) yield
$$\big(\Theta_{jk}^{g}u,v\big)_g=v^*g\Theta_{jk}^{g}u=-v^*gg^{-1}\Theta_{kj}^{g^{-1}}gu=v^*(\Theta_{kj}^g)^*gu=\big(u,\Theta_{kj}^gv\big)_g.$$
This proves (i).

To prove (ii), we note that by (\ref{eq:1}) and property (i)
{\setlength\arraycolsep{2pt}
\begin{eqnarray}
\big(\Theta_{jk}^{g^{-1}}u,v\big)_{g^{-1}}&=&v^*g^{-1}\Theta_{jk}^{g^{-1}}u=-(g^{-1}v)^*(\Theta_{jk}^{g})^*g(g^{-1}u)=\nonumber\\
&=&-\big(g^{-1}u,\Theta_{jk}^gg^{-1}v\big)_g=-\big(\Theta_{kj}^gg^{-1}u,g^{-1}v\big)_g.\nonumber
\end{eqnarray}}
\!\!In particular, if we take any $n$-tuple of vectors $\{u_j\}_{j=1}^n\subset\C^r$ and let $v_j=g^{-1}u_j$, we get that
$$\sum_{j,k=1}^n\big(\Theta_{jk}^{g^{-1}}u_j,u_k\big)_{g^{-1}}=-\sum_{j,k=1}^n\big(\Theta_{kj}^gv_j,v_k\big)_g.$$
Thus, we see that in the Griffiths setting, if $g$ is log concave, then $g^{-1}$ is log convex, but since the indices on the right hand side are switched, the same conclusion can not be made in the Nakano case. In example \ref{ex:4} below we will study a metric, $g$, such that $g^{-1}$ is Nakano log concave while $g$ is neither log concave nor log convex in the sense of Nakano.

To prove (iii), let
$$S^g(w):=\sum_{j,k=1}^n\Theta^g_{jk}w_j\bar{w}_k,$$
where $w=(w_1,\ldots,w_n)\in\C^n$. Then, saying that $g$ is log concave in the sense of Griffiths, is equivalent to saying that for each vector $w$, the matrix-valued function $S^g(w)$ is negative semidefinite.

Now it is a well-known fact from linear algebra, (see e.g. \cite{L}, Theorem 4, Chapter 9), that
\be\label{eq:LA}
\frac{\partial}{\partial x_k}\log\det g=tr\Big(g^{-1}\frac{\partial g}{\partial x_k}\Big).
\ee
Hence
$$\sum_{j,k=1}^nw_j\frac{\partial^2\log\det g}{\partial x_j\partial x_k}\bar{w}_k=\sum_{j,k=1}^nw_jtr\big(\Theta_{jk}^g\big)\bar{w}_k=tr\big(S^g(w)\big),$$
and since $S^g(w)$ is negative semidefinite, $tr(S^g(w))$ will be negative as well. This shows that $\det g$ is a log concave function and finishes the proof of the proposition.

Finally, the proof of (iv) is an immediate consequence of the arguments in examples \ref{ex:2} and \ref{ex:3}.
\end{proof}

\begin{ex}\label{ex:4}
Let $g:\R^n\to\C^{n\times n}$ be a metric, and assume that for some $x_0\in\R^n$, $g(x_0)=I$ and
\be\label{eq:counter_ex}
\Theta_{jk}^g(x_0)=\left\{\begin{array}{cl}
1 & \textrm{row $j$, column $k$}, \\
0 & \textrm{otherwise}. 
\end{array}\right.
\ee
(An explicit example of such a metric when $n=2$ will be given below.)

For two vectors $v,w\in\C^n$, it is then straightforward to see that
$$\big(\Theta_{jk}^gv,w\big)_g=v_k\bar{w}_j,$$
at $x_0$. Hence, for any $n$-tuple $\{u_j\}_{j=1}^n\subset\C^n$ with $u_j=(u_{j1},\ldots,u_{jn})$, we get that
\be\label{eq:3}
\sum_{j,k=1}^n\big(\Theta_{jk}^gu_j,u_k\big)_g=\sum_{j,k=1}^nu_{jk}\bar{u}_{kj}.
\ee
This expression is neither positive nor negative.

Now let us study the inverse, or dual, metric $g^{-1}$. Here we obviously still have that $g^{-1}(x_0)=I$, and from (\ref{eq:1}) above, we also have that
$$\Theta^{g^{-1}}_{jk}(x_0)=-\big(\Theta^g_{jk}\big)^*(x_0)=\left\{\begin{array}{cl}
-1 & \textrm{row $k$, column $j$}, \\
0 & \textrm{otherwise}, 
\end{array}\right.$$
and so in this case
$$\sum_{j,k=1}^n\big(\Theta_{jk}^{g^{-1}}u_j,u_k\big)_{g^{-1}}=-\sum_{j,k=1}^nu_{jj}\bar{c}_{kk}=-\Big|\sum_{j=1}^nu_{jj}\Big|^2\leq0.$$
Thus, in the sense of Nakano, $g^{-1}$ is log concave while $g$ is neither log convex nor log concave.

Note that we in this way, as a bonus, have found a metric that is log convex in the sense of Griffiths, but not in the sense of Nakano. Since $g^{-1}$ is also Griffiths log concave, by Proposition \ref{prop:1} (ii), $g$ will be Griffiths log convex. The reason behind this is that in the Griffiths setting, each $u_{jk}$ in (\ref{eq:3}) is separable. More precisely, each $u_j$ will be replaced by $uv_j$, with $u,v\in\C^n$, and we will get
$$\sum_{j,k=1}^n\big(\Theta_{jk}^gu,u\big)_gv_j\bar{v}_k=\sum_{j,k=1}^nu_j\bar{u}_kv_j\bar{v}_k=\Big|\sum_{j=1}^nu_jv_j\Big|^2\geq0.$$

Finally, let us give an explicit example of a metric with $\Theta$ satisfying (\ref{eq:counter_ex}). Namely let $g:\R^2\to\R^{2\times 2}$ be the metric, whose inverse metric $g^{-1}$ is given by,
$$g^{-1}(x_1,x_2)=I+x_1\left( \begin{array}{cc}
1 & 0 \\
0 & 0 
\end{array} \right)+x_2\left( \begin{array}{cc}
0 & 1 \\
1 & 0 
\end{array} \right)+\frac{x_2^2}{2}\left( \begin{array}{cc}
1 & 0 \\
0 & 0 
\end{array} \right).$$
It is clear that $g(0)=g^{-1}(0)=I$, and using that
$$\Theta_{jk}^{g^{-1}}=-g\frac{\partial g^{-1}}{\partial x_k}g\frac{\partial g^{-1}}{\partial x_j}+g\frac{\partial^2 g^{-1}}{\partial x_j\partial x_k}$$
we get
$$\begin{array}{l}
\Theta_{11}^{g^{-1}}(0)=-\left( \begin{array}{cc}
1 & 0 \\
0 & 0 
\end{array} \right)^2=\left( \begin{array}{cc}
-1 & 0 \\
0 & 0 
\end{array} \right),\vspace{0.1cm}\\
\Theta_{12}^{g^{-1}}(0)=-\left( \begin{array}{cc}
0 & 1 \\
1 & 0 
\end{array} \right)\left( \begin{array}{cc}
1 & 0 \\
0 & 0 
\end{array} \right)=\left( \begin{array}{cc}
0 & 0 \\
-1 & 0 
\end{array} \right),\vspace{0.1cm}\\
\Theta_{21}^{g^{-1}}(0)=-\left( \begin{array}{cc}
1 & 0 \\
0 & 0 
\end{array} \right)\left( \begin{array}{cc}
0 & 1 \\
1 & 0 
\end{array} \right)=\left( \begin{array}{cc}
0 & -1 \\
0 & 0 
\end{array} \right),\vspace{0.1cm}\\
\Theta_{22}^{g^{-1}}(0)=-\left( \begin{array}{cc}
0 & 1 \\
1 & 0 
\end{array} \right)^2+\left( \begin{array}{cc}
1 & 0 \\
0 & 0 
\end{array} \right)=\left( \begin{array}{cc}
0 & 0 \\
0 & -1 
\end{array} \right).
\end{array}$$
Altogether we see that
$$\Theta^{g^{-1}}_{jk}(0)=\left\{\begin{array}{cl}
-1 & \textrm{row $k$, column $j$}, \\
0 & \textrm{otherwise}, 
\end{array}\right.$$
and hence $\Theta^g_{jk}(0)$ satisfies (\ref{eq:counter_ex}).\qed
\end{ex}

Constructing metrics that are Nakano log concave, which is needed for Theorem \ref{thm:main}, is a priori not an easy task. In the Griffiths case, however, the situation is quite different. In Proposition \ref{prop:altG} below, we will give an alternative characterization of Griffiths log convex metrics, which combined with Proposition \ref{prop:1} (ii), facilitates the construction of Griffiths log concave metrics considerably.

Given this fact it is then natural to ask if it is possible to construct a metric that is log concave in the sense of Nakano, out of a given metric which is log concave in the sense of Griffiths. In the complex-variable setting, a very elegant solution to this problem is provided by a celebrated theorem due to Demailly and Skoda (\cite{DS}). Reformulated to our real-variable setting, this theorem states the following.

\begin{thm}\label{thm:DS}
Let $g:\R^n\to\C^{r\times r}$ be a metric. If $g$ is log concave in the sense of Griffiths, then $g\cdot\det g$ is log concave in the sense of Nakano.
\end{thm}

Since this theorem holds in the complex-variable setting, by interpreting the metric $g:\R^n\to\C^{r\times r}$ as a metric $h:\C^n_z\to\C^{r\times r}$ that does not depend on the real part of $z$, it will also hold in the real-variable setting. Nevertheless, for the sake of completeness, we have still chosen to include a proof. The argument is the same as the one in \cite{D}, Theorem 8.2, Chapter VII; it is only adapted to our real-valued setting and our matrix oriented notation.

The proof of Theorem \ref{thm:DS} relies heavily on the following discrete Forurier transform type of lemma.

\begin{lma}\label{lma:1}
Let $q\geq3$ be an integer and let $x,y\in\C^r$ be two vectors. Let furthermore
$$U_q^r:=\Big\{(e^{2\pi ik_1/q},\ldots,e^{2\pi ik_r/q});\ k_1,\ldots,k_r\in\{0,\ldots,q-1\}\Big\}$$
and, for $\sigma\in U_q^r$, put
$$(x,\sigma):=\sum_{\mu=1}^rx_\mu\bar\sigma_\mu,\qquad(y,\sigma):=\sum_{\mu=1}^ry_\mu\bar\sigma_\mu.$$
Then, for any pair of integers $(\alpha,\beta)$ with $1\leq\alpha,\beta\leq r$, the following identity holds:
$$\frac{1}{q^r}\sum_{\sigma\in U_q^r}(x,\sigma)\overline{(y,\sigma)}\sigma_\alpha\bar\sigma_\beta=\left\{\begin{array}{cl}
x_\alpha\bar y_\beta & \textrm{if }\alpha\neq\beta, \\
\sum_{\mu=1}^rx_\mu\bar y_\mu & \textrm{if }\alpha=\beta. 
\end{array}\right.$$
\end{lma}
For a proof of this lemma see \cite{D} Lemma 8.3, Chapter VII. (This might seem to contradict what we just said about a complete proof, but the notation in the proof of the lemma in \cite{D} is standard, while the notation in the proof of the theorem is completely different from ours.)

\begin{proof}[Proof of Theorem \ref{thm:DS}]
Let $\tilde{g}=g\det g$. Then, using (\ref{eq:LA}) above, a short computation shows that
{\setlength\arraycolsep{2pt}
\begin{eqnarray}
\Theta_{jk}^{\tilde{g}}&=&\frac{\partial}{\partial x_j}\left(\frac{g^{-1}}{\det g}\frac{\partial}{\partial x_k}(g\det g)\right)=\Theta_{jk}^g+\frac{\partial^2 \log\det g}{\partial x_j\partial x_k}\cdot I=\nonumber\\
&=&\Theta_{jk}^g+tr\big(\Theta_{jk}^g\big)\cdot I.\nonumber
\end{eqnarray}}
\!\!Hence, what we want to show is that for arbitrary vectors $u\in\C^r$, $v\in\C^n$, and $\{u_j\}_{j=1}^n\subset\C^r$, if
$$\sum_{j,k=1}^n\big(\Theta_{jk}^gu,\bar u\big)_gv_j\bar v_k\leq0,$$
then this implies that
$$\sum_{j,k=1}^n\Big(\big(\Theta_{jk}^g+tr(\Theta_{jk}^g)\big)u_j,u_k\Big)_{g\det g}\leq0.$$

First of all we note that, since this statement is pointwise, we can without any loss of generality assume that we are working in an ortonormal basis for $g$, so that $g=I$.

Now let $q$ and $U_q^r$ be as in Lemma \ref{lma:1}, and let us study the expression
\be\label{eq:GN}
\frac{1}{q^r}\sum_{\sigma\in U_q^r}\sum_{j,k=1}^n\big(\Theta_{jk}^g\sigma,\sigma\big)(u_j,\sigma)\overline{(u_k,\sigma)},
\ee
where all inner products are with respect to the euclidean metric. Since $g$ is log concave in the sense of Griffiths, the inner sum, and hence the entire expression, is negative. On the other hand, we can switch the order of summation and apply Lemma \ref{lma:1}. However, to avoid proliferation of indices, we first do this with $\Theta_{jk}^g$
replaced by an arbitrary matrix $A\in\R^{r\times r}$, and $u_j, u_k$ replaced by the (real) vectors $x$ and $y$ respectively, i.e. we study the expression
$$\frac{1}{q^r}\sum_{\sigma\in U_q^r}(A\sigma,\sigma)(x,\sigma)\overline{(y,\sigma)}.$$

If we write
$$(A\sigma,\sigma)=\sum_{\alpha,\beta=1}^rA_{\alpha\beta}\sigma_\alpha\bar\sigma_\beta,$$
we get that
$$\frac{1}{q^r}\sum_{\sigma\in U_q^r}(A\sigma,\sigma)(x,\sigma)\overline{(y,\sigma)}=\sum_{\alpha,\beta=1}^rA_{\alpha\beta}\Big(\frac{1}{q^r}\sum_{\sigma\in U_q^r}(x,\sigma)\overline{(y,\sigma)}\sigma_\alpha\bar\sigma_\beta\Big).$$
Thus, applying Lemma \ref{lma:1} and rewriting slightly, we end up with
{\setlength\arraycolsep{2pt}
\begin{eqnarray}
&&\frac{1}{q^r}\sum_{\sigma\in U_q^r}(A\sigma,\sigma)(x,\sigma)\overline{(y,\sigma)}=\sum_{\alpha\neq\beta}A_{\alpha\beta}x_\alpha\bar y_\beta+tr(A)\sum_{\alpha=1}^rx_\alpha\bar y_\alpha=\nonumber\\
&&\qquad\qquad=\Big(\big(A+tr(A)\big)x,y\Big)-\sum_{\alpha=1}^rA_{\alpha\alpha}x_\alpha\bar y_\alpha.\nonumber
\end{eqnarray}}

Returning to the original expression (\ref{eq:GN}), we see that this, together with the log concavity in the sense of Griffiths, implies that
{\setlength\arraycolsep{2pt}
\begin{eqnarray}
0&\geq&\frac{1}{q^r}\sum_{\sigma\in U_q^r}\sum_{j,k=1}^n\big(\Theta_{jk}^g\sigma,\sigma\big)(u_j,\sigma)\overline{(u_k,\sigma)}=\nonumber\\
&&\qquad=\sum_{j,k=1}^n\Big(\big(\Theta_{jk}^g+tr(\Theta_{jk}^g)\big)u_j,u_k\Big)-\sum_{j,k=1}^n\sum_{\alpha=1}^r\big(\Theta^g_{jk}\big)_{\alpha\alpha}u_{j\alpha}\bar u_{k\alpha}.\nonumber
\end{eqnarray}}
\!\!where $u_j=(u_{j1},\ldots,u_{jr})$. The last term can now be rewritten as
$$-\sum_{j,k=1}^n\sum_{\alpha=1}^r\big(\Theta^g_{jk}\big)_{\alpha\alpha}u_{j\alpha}\bar u_{k\alpha}=-\sum_{\alpha=1}^r\sum_{j,k=1}^n\big(\Theta^g_{jk}e_\alpha,e_\alpha\big)u_{j\alpha}\bar u_{k\alpha}\geq0,$$
where $\{e_\alpha\}_{\alpha=1}^r$ is an orthonormal basis for $g$, and the positivity once again follows from the log concavity of $g$ in the sense of Griffiths. Hence,
$$\sum_{j,k=1}^n\Big(\big(\Theta_{jk}^g+tr(\Theta_{jk}^g)\big)u_j,u_k\Big)\leq0,$$
which is what we wanted to prove.
\end{proof}

Finally, we have the alternative characterization of Griffiths log convexity alluded to above. In the complex-variable setting, it says the following, (see e.g. \cite{B}, section 2, for a proof).

\begin{prop}\label{prop:altG}
Let $h:\C^n\to\C^{r\times r}$ be a hermitian metric. Then $h$ is negatively curved in the sense of Griffiths if and only if
$$\log\|u\|^2_h$$
is plurisubharmonic for every holomorphic function $u$.
\end{prop}

We have unfortunately not managed to find any real-variable analogue of this property. The main difficulty have been to find an appropriate replacement for the holomorphic functions $u$.

Nevertheless, Proposition \ref{prop:altG} is still very useful even in the real-variable setting, (if we regard our metrics as hermitian metrics as described in the beginning of this section). For example, combined with Proposition \ref{prop:1} (ii), these results immediately yield that the sum of two Griffiths log convex metrics is also Griffiths log convex; a fact that is not very easy to verify by direct computation.

\section{The proof of Theorem \ref{thm:PW}}
\noindent As mentioned in the introduction, the scalar-valued case of Theorem \ref{thm:PW} was shown by Jakob Hultgren in his master's thesis. The main ideas in our proof are the same as in the thesis, but since this thesis is not properly published, we have chosen to include a complete proof here as well.

We start by proving that if $f\in L^2(\tilde{g})$, i.e. $f:\R^n\to\C^r$ is such that
\be\label{eq:f}
\int_{\R^n}\|f(\xi)\|^2_{\tilde{g}(\xi)}dV(\xi)<\infty
\ee
then the function $F:\C^n\to\C^r$ defined through
\be\label{eq:Ff}
F(z):=\int_{\R^n}f(\xi)e^{-i\xi\cdot z}dV(\xi)
\ee
is in $A^2(g)$, i.e. $F$ is holomorphic and
\be\label{eq:F}
\int_{\C^n}\|F(z)\|^2_{g(y)}dV(z)<\infty.
\ee

To show that each vector element $F_j$ of $F$ is holomorphic, we note that it is enough to show that $F_j$ restricted to any complex line in $\C^n$ is holomorphic. We apply Morera's theorem to prove this latter statement. Hence, we let $\gamma$ denote a triangle in a complex line in $\C^n$ and study
\be\label{eq:M}
\int_{\gamma}F_j(z)dz=\int_{\gamma}\bigg(\int_{\R^n}f_j(\xi)e^{-i\xi\cdot z}dV(\xi)\bigg)dz.
\ee
As $e^{-i\xi\cdot z}$ is holomorphic in $z$, we would be done if we could apply Fubini's theorem. We now turn to justifying this, which turns out to be much more involved than one might expect at first sight.

To apply Fubini's theorem in (\ref{eq:M}) we need to show that $f_j(\xi)e^{-i\xi\cdot z}$ is integrable with respect to both $z$ and $\xi$. However, since $\gamma$ is compact, it suffices to prove this locally in $z$. Hence, we will show that for any fix $z_0=x_0+iy_0\in\C^n$, there exists a function $G_j\in L^1(\R^n)$ such that
$$|f_j(\xi)e^{-i\xi\cdot z}|\leq G_j(\xi)$$
in a neighborhood of $z_0$.

For this, we start by noting that
$$|f_j(\xi)e^{-i\xi\cdot z}|=|f_j(\xi)|e^{\xi\cdot y}=|f_j(\xi)|e^{\xi\cdot(y-y_0)}e^{\xi\cdot y_0}\leq|f_j(\xi)|e^{\xi\cdot y_0}e^{\varepsilon|\xi|/2}$$
for all $z=x+iy\in\C^n$ such that $|y-y_0|<\varepsilon/2$, for some $\varepsilon>0$. We set $G_j(\xi):=|f_j(\xi)|e^{\xi\cdot y_0}e^{\varepsilon|\xi|/2}$.

To show that $G_j$ is integrable, we begin by showing that $f\in L^2(\tilde{g})$ implies that
\be\label{eq:ae}
\int_{\R^n}|f_j(\xi)|^2e^{2\xi\cdot\tilde{y}}dV(\xi)<\infty,
\ee
for all $j=1,\ldots,r$ and every $\tilde{y}\in\R^n$.

For this we let $\tilde{y}\in\R^n$ be arbitrary and note that since $g$ is assumed to be continuous and $g(\tilde{y})$ is strictly positive definite, there exists $\tilde{\varepsilon},\delta>0$, such that $y\in B_\delta(\tilde{y})$ implies that $g(y)\geq\tilde{\varepsilon}I$. Using this in the definition of $\tilde{g}$ yields
{\setlength\arraycolsep{2pt}
\begin{eqnarray}
\tilde{g}(\xi)=(2\pi)^n\int_{\R^n}e^{2\xi\cdot y}g(y)dV(y)&\geq&(2\pi)^n\int_{B_\delta(\tilde{y})}e^{2\xi\cdot y}g(y)dV(y)\geq\nonumber\\
&\geq&(2\pi)^n\tilde{\varepsilon}I\int_{B_\delta(\tilde{y})}e^{2\xi\cdot y}dV(y).\nonumber
\end{eqnarray}}
\!\!By normalizing, applying Jensen's inequality, and using the definition of barycenter, we can further rewrite this as
{\setlength\arraycolsep{2pt}
\begin{eqnarray}
\tilde{g}(\xi)&\geq&(2\pi)^n\tilde{\varepsilon}|B_\delta(\tilde{y})|I\int_{B_\delta(\tilde{y})}e^{2\xi\cdot y}\frac{dV(y)}{|B_\delta(\tilde{y})|}\geq\nonumber\\
&\geq&(2\pi)^n\tilde{\varepsilon}|B_\delta(\tilde{y})|I\exp\bigg(2\int_{B_\delta(\tilde{y})}\xi\cdot y\frac{dV(y)}{|B_\delta(\tilde{y})|}\bigg)\geq(2\pi)^n\tilde{\varepsilon}|B_\delta(\tilde{y})|e^{2\xi\cdot\tilde{y}}I.\nonumber
\end{eqnarray}}
\!\!Thus,
{\setlength\arraycolsep{2pt}
\begin{eqnarray}
\int_{\R^n}\|f(\xi)\|^2_{\tilde{g}(\xi)}dV(\xi)&\geq&(2\pi)^n\tilde{\varepsilon}|B_\delta(\tilde{y})|\int_{\R^n}e^{2\xi\cdot \tilde{y}}\|f(\xi)\|^2_{I}dV(\xi)=\nonumber\\
&=&(2\pi)^n\tilde{\varepsilon}|B_\delta(\tilde{y})|\sum_{j=1}^r\int_{\R^n}|f_j(\xi)|^2e^{2\xi\cdot\tilde{y}}dV(\xi).
\end{eqnarray}}
\!\!This proves (\ref{eq:ae}).

Now decompose $\R^n$ into a finite number of cones, $\{\Gamma_k\}_{k=1}^m$, with vertex at the origin. This decomposition should be made in such a way that if $z_1,z_2\in\Gamma_k$, then
\be\label{eq:angle}
z_1\cdot z_2>\frac{|z_1||z_2|}{2}.
\ee
This simply means that the cosinus of the angle between $z_1$ and $z_2$ must be greater than $1/2$, i.e. the angle between any two points in $\Gamma_k$ must be less than $\pi/3$.

Hence, if we choose $y_k\in\Gamma_k$ with $|y_k|\geq1$, then by (\ref{eq:ae}) and (\ref{eq:angle})
{\setlength\arraycolsep{2pt}
\begin{eqnarray}\label{eq:finite}
&&\int_{\Gamma_k}|f_j(\xi)|^2e^{2\xi\cdot y_0}e^{|\xi|}dV(\xi)\leq\int_{\Gamma_k}|f_j(\xi)|^2e^{2\xi\cdot y_0}e^{|\xi||y_k|}dV(\xi)\leq\nonumber\\
&&\qquad\leq\int_{\Gamma_k}|f_j(\xi)|^2e^{2\xi\cdot(y_0+y_k)}dV(\xi)\leq\int_{\R^n}|f_j(\xi)|^2e^{2\xi\cdot(y_0+y_k)}dV(\xi)<\infty.
\end{eqnarray}}

By the Cauchy-Schwarz inequality
{\setlength\arraycolsep{2pt}
\begin{eqnarray}
&&\int_{\R^n}G_j(\xi)dV(\xi)=\int_{\R^n}|f_j(\xi)|e^{\xi\cdot y_0}e^{\varepsilon|\xi|}e^{-\varepsilon|\xi|/2}dV(\xi)\leq\nonumber\\
&&\qquad\leq\bigg(\int_{\R^n}|f_j(\xi)|^2e^{2\xi\cdot y_0}e^{2\varepsilon|\xi|}dV(\xi)\bigg)^{1/2}\bigg(\int_{\R^n}e^{-\varepsilon|\xi|}dV(\xi)\bigg)^{1/2}.\nonumber
\end{eqnarray}}
\!\!The second factor here is finite for all $\varepsilon>0$, and by decomposing $\R^n$, choosing $\varepsilon=1/2$, and using (\ref{eq:finite}), we get that,
$$\int_{\R^n}|f_j(\xi)|^2e^{2\xi\cdot y_0}e^{|\xi|}dV(\xi)=\sum_{k=1}^m\int_{\Gamma_k}|f_j(\xi)|^2e^{2\xi\cdot y_0}e^{|\xi|}dV(\xi)<\infty.$$
Thus $G_j$ is integrable, which shows that $F_j$ is holomorphic for any $j=1,\ldots,r$.

It remains to show that $F\in A^2(g)$, i.e. square integrable with respect to $g$. By definition, for fix $y\in\R^n$, $F_y(x):=F(x+iy)$ is just the Fourier transform of $e^{\xi\cdot y}f(\xi)$. By (\ref{eq:ae}), $e^{\xi\cdot y}f_j(\xi)\in L^2(\R^n)$ for any $j=1,\ldots,r$, and so by Parseval's formula, $F_{j,y}\in L^2(\R^n)$ as well. In particular then, for any $j,k=1,\ldots,r$,
$$f_j(\xi)\bar{f}_k(\xi)e^{2\xi\cdot y}\in L^1(\R^n)\quad\textrm{ and }\quad F_{j,y}(x)\bar{F}_{k,y}(x)\in L^1(\R^n).$$ 
Hence, we can exchange summation and integration, and use Plancherel's formula together with the Fubini-Tonelli theorem to deduce that
{\setlength\arraycolsep{2pt}
\begin{eqnarray}\label{eq:Pars}
&&\int_{\C^n}\|F(z)\|^2_{g(y)}dV(z)=\!\!\sum_{j,k=1}^r\int_{\R^n}\bigg(\int_{\R^n}\!\!F_{j,y}(x)\bar{F}_{k,y}(x)dV(x)\bigg)g_{jk}(y)dV(y)=\nonumber\\
&&\qquad\qquad=(2\pi)^n\sum_{j,k=1}^r\int_{\R^n}\bigg(\int_{\R^n}f_j(\xi)\bar{f}_k(\xi)e^{2\xi\cdot y}dV(\xi)\bigg)g_{jk}(y)dV(y)=\nonumber\\
&&\qquad\qquad=(2\pi)^n\int_{\R^n}\bigg(\int_{\R^n}e^{2\xi\cdot y}\|f(\xi)\|^2_{g(y)}dV(\xi)\bigg)dV(y)=\nonumber\\
&&\qquad\qquad=\sum_{j,k=1}^r\int_{\R^n}f_j(\xi)\bar{f}_k(\xi)\bigg((2\pi)^n\int_{\R^n}e^{2\xi\cdot y}g_{jk}(y)dV(y)\bigg)dV(\xi)=\nonumber\\
&&\qquad\qquad=\sum_{j,k=1}^r\int_{\R^n}f_j(\xi)\bar{f}_k(\xi)\tilde{g}_{jk}(\xi)dV(\xi)=\int_{\R^n}\|f(\xi)\|^2_{\tilde{g}(\xi)}dV(\xi).
\end{eqnarray}}
\!\!This shows that $f\in L^2(\tilde{g})$ implies that $F\in A^2(g)$.

We now turn to the converse problem. Hence, we assume that $F\in A^2(g)$ is given and want to construct a function $f\in L^2(\tilde{g})$ such that (\ref{eq:Ff}) holds.

As before, set $F_y(x):=F(x+iy)$ for fix $y\in\R^n$. We claim that $f(\xi):=\F^{-1}(F_0)(\xi)$, the inverse Fourier transform of $F_0$, which seems natural considering (\ref{eq:Ff}). Our plan is to prove this in two steps. First we must show that $F_0$ is in $L^2$, so that the inverse Fourier transform is well-defined. After this we show that $f:=\F^{-1}(F_0)\in L^2(\tilde{g})$. It then follows from the first part of this proof that the function
$$G(z)=\int_{\R^n}f(\xi)e^{-i\xi\cdot z}dV(\xi)$$
is holomorphic in $z$. Since $F(x)=G(x)$ for $x\in\R^n$, by analytic continuation
$$F(z)=\int_{\R^n}f(\xi)e^{-i\xi\cdot z}dV(\xi)$$
everywhere, and we are done. 

The following lemma will be central.

\begin{lma}\label{lma:4.1}
Let $G:\C^n\to\R$ be a continuous, subharmonic, and strictly positive function. If
\be\label{eq:lma}
\int_{\R^n\times K}G(x+iy)dV(x,y)<\infty
\ee
for all compact sets $K\subset\R^n$, then
$$\int_{\R^n}G(x+iy)dV(x)$$
is continuous in $y$.
\end{lma}

\begin{proof}
Fix $y_0\in\R^n$ arbitrarily. We are going to use dominated convergence and so we want to find a function $G_{y_0}\in L^1(\R^n)$, such that
\be\label{eq:dom}
G(x,y)\leq G_{y_0}(x)
\ee
for all $y$ in a neighborhood of $y_0$.

Now let $B_R^\R(y_0)\subset\R^n$ be a, (real), ball of radius $R>0$, centered at $y_0$. By the submean-inequality for subharmonic functions, for some, (complex), ball $B_{\tilde{r}}^\C\subset\C^n$ of radius $\tilde{r}>0$, centered at the origin
{\setlength\arraycolsep{2pt}
\begin{eqnarray}
G(x+iy)&\leq&\frac{1}{|B_{\tilde{r}}^\C|}\int_{B^\C_{\tilde{r}}}G(z+w)dV(w)\leq\nonumber\\
&\leq&\frac{1}{|B_{\tilde{r}}^\C|}\int_{B^\R_{\tilde{r}}}\bigg(\int_{B^\R_{\tilde{r}}}G\big((x+s)+i(y+t)\big)dV(t)\bigg)dV(s)\nonumber
\end{eqnarray}}
\!\!where $|B_{\tilde{r}}^\C|$ denotes the Lebesgue measure of the ball, and we have used the fact that $B_{\tilde{r}}^\C\subset B_{\tilde{r}}^\R\times B_{\tilde{r}}^\R$. Hence by choosing $R>0$ such that $B_{\tilde{r}}^\R(y)\subset\bar{B}_R^\R(y_0)$, we get that
$$G(x+iy)\leq\frac{1}{|B_{\tilde{r}}^\C|}\int_{B^\R_{\tilde{r}}}\bigg(\int_{\bar{B}^\R_R}G\big((x+s)+i(y_0+t)\big)dV(t)\bigg)dV(s):=G_{y_0}(x).$$
By construction, then, (\ref{eq:dom}) holds. Furthermore, by using the Fubini-Tonelli theorem, (\ref{eq:lma}), and a linear change of variables, we also have that
$$\int_{\R^n}\!\!G_{y_0}(x)dV(x)\!=\!\frac{1}{|B_{\tilde{r}}^\C|}\!\int_{B^\R_{\tilde{r}}}\bigg(\int_{\R^n\times\bar{B}^\R_R}\!\!\!\!\!\!\!\!G\big((x+s)+i(y_0+t)\big)dV(x,t)\bigg)dV(s)\!<\!\infty.$$
\end{proof}

We apply this lemma with $G(x,y)=|F_j(x+iy)|^2$ for $j=1,\ldots,r$, which are subharmonic, (in fact plurisubharmonic), as $F$ is holomorphic. Also since the metric $g$ only depends on $y$, is strictly positive definite everywhere, and is continuous, it will be locally bounded from below. Hence, $F\in A^2(g)$ implies that
$$\int_{\R^n\times K}|F_j(x+iy)|^2dV(x,y)<\infty,$$
for all compact sets $K\subset\R^n$.

Thus, Lemma \ref{lma:4.1} yields that
$$\int_{\R^n}|F_j(x+iy)|^2dV(x)=\int_{\R^n}|F_{j,y}(x)|^2dV(x)$$
is continuous in $y$, (in particular finite), and so $F_y\in L^2(\R^n)$ for all $y\in\R^n$. This proves that $f(\xi):=\F^{-1}(F_0)(\xi)$ is well-defined.

It remains to show that $f\in L^2(\tilde{g})$. By (\ref{eq:Pars}) this will follow if we can establish that for all $\xi\in\R^n$,
\be\label{eq:main}
\F^{-1}(F_y)(\xi)=\F^{-1}(F_0)(\xi)e^{\xi\cdot y}=f(\xi)e^{\xi\cdot y}.
\ee
This is a bit tricky. An important observation is that (\ref{eq:main}) holds if and only if
$$\F^{-1}\big((F\circ T)_y\big)(\xi)=\F^{-1}\big((F\circ T)_0\big)(\xi)e^{\xi\cdot T(y)}$$
for any invertible linear mapping $T:\C^n\to\C^n$. This is straightforward to verify, and it is important since it allows us to choose coordinates so that $y=(y_1,0,\ldots,0)$, thereby reducing the proof of (\ref{eq:main}) to the one-dimensional case.

Now fix $x_2,\ldots,x_n\in\R$, let $z=x_1+iy_1\in\C$, $\xi\in\R^n$ and set
$$G_j(z)=F_j(z,x_2,\ldots,x_n)e^{i\xi_1z}e^{i\sum_{j=2}^n\xi_jx_j}.$$
Let $\gamma_1$ be the $x$-axis in $\C$, let $\gamma_2$ be an arbitrary horizontal line in $\C$, and let $S$ denote the strip between these two lines. Furthermore, let $\chi:\R\to\R$ be a cut-off function which is equal to one on the closed unit ball in $\R$, (i.e. $\chi\in C^{\infty}_c(\R)$, $0\leq\chi\leq1$ and $\chi=1$ on $[-1,1]$), and set $\chi_R(x):=\chi(x/R)$, with $R>0$. Then, by Stokes' theorem,
{\setlength\arraycolsep{2pt}
\begin{eqnarray}\label{eq:problem}
&&\int_{\gamma_1}\chi_R\big(\Re(z)\big)G_j(z)dz-\int_{\gamma_2}\chi_R\big(\Re(z)\big)G_j(z)dz=\\
&&=\!\!\int_{S}\!\frac{\partial}{\partial\bar{z}}\big(\chi_R\big(\Re(z)\big)G_j(z)\big)dV(z)\!=\!\!\!\int_{S}\!\bigg(\frac{\partial}{\partial\bar{z}}\!\chi_R\big(\Re(z)\big)\!\!\bigg)G_j(z)dV(z),\nonumber
\end{eqnarray}}
\!\!as $G_j$ is holomorphic.

Since the function $\phi_R(x):=\chi_R(x)-1$ is decreasing in $R$, by the monotone convergence theorem, $\phi_R\to0$ in $L^2(\R)$ as $R\to\infty$. At the same time, if we assume that $\gamma_2$ has the parametrization $x+iy$, with $x\in\R$ and $y>0$, then
{\setlength\arraycolsep{2pt}
\begin{eqnarray}
&&\int_{S}\Big|\frac{\partial}{\partial\bar{z}}\chi_R\big(\Re(z)\big)\Big|^2dV(z)=\frac{y}{4}\int_{\R}\Big|\frac{\partial}{\partial x}\chi_R(x)\Big|^2dx=\nonumber\\
&&\qquad\qquad\qquad\qquad\qquad=\frac{y}{4R^2}\int_{\R}\Big|\chi'(\frac{x}{R})\Big|^2dx=\frac{y}{4R}\int_{\R}|\chi'(t)|^2dt,\nonumber
\end{eqnarray}}
\!\!and so $(\chi_R)'_{\bar{z}}\to0$ in $L^2(S)$, as $R\to\infty$.

On the other hand, one can show that
\be\label{eq:GL2}
\int_\R\big|G_j(x_1+iy_1)\big|^2dx_1
\ee
is continuous in $y_1$, (in particular, $G_{j,y_1}\in L^2(\R)$ for any fix $y_1\in\R$). This follows from basically the same argument as in Lemma \ref{lma:4.1}, but this time using the $pluri$subharmonicity of $|F|^2$. The continuity of (\ref{eq:GL2}) then also implies that $G_j\in L^2(S)$.

Hence, letting $R\to\infty$ in (\ref{eq:problem}), these facts, together with the Cauchy-Schwarz inequality yield that,
$$\int_{\gamma_1}G(z)dz=\int_{\gamma_2}G(z)dz,$$
which in turn is equivalent to
$$\int_\R F(x_1+iy_1,x_2,\ldots,x_n)e^{i\xi\cdot x}e^{-\xi_1y_1}dx_1=\int_\R F(x_1,x_2,\ldots,x_n)e^{i\xi\cdot x}dx_1.$$
Integrating this identity with respect to $x_2,\ldots,x_n$, and choosing coordinates so that $y=(y_1,0,\ldots,0)$ as before, we, at last, get that
$$\F^{-1}(F_y)(\xi)e^{-\xi\cdot y}=\F^{-1}(F_0)(\xi),$$
which is what we wanted to show.

\section{The proof of Theorem \ref{thm:B}}
\noindent We now turn to the proof of Theorem \ref{thm:B}. The plan is to use the same approach as in the scalar valued case, (\cite{B} Theorem 1.1), namely the Griffiths subbundle formula and H\"ormander $L^2$-estimates for the $\dbar$-equation. However, since in the vector-valued case, these $L^2$-estimates differ a bit from the ones in the scalar-valued setting, the second part of the proof will be somewhat different from the proof in \cite{B}. 

As our main goal in this paper is the proof of Theorem \ref{thm:main}, we have chosen not to have a lengthy review of the basic concepts of (infinite rank) vector bundles. For this, we refer to section 2 of \cite{B}.

\begin{proof}[Proof of Theorem \ref{thm:B}]
Since the first part of the proof is rather similar to the corresponding argument in \cite{B}, our treatment of this part will be a little sketchy.

We want to show that given a hermitian metric $h:\C^n_z\times\C^m_w\to\C^{r\times r}$, which is positively curved in the sense of Nakano, the trivial, (infinite rank), holomorphic vector bundle $(E,\|\cdot\|_w)$ over $\C^m$, defined through $E_w:=A^2_w(h)$ as in (\ref{eq:fiber}), is Nakano positive as well.

We begin by observing that $E$, is a subbundle of the trivial, holomorphic vector bundle $(G,\|\cdot\|_w)$, where
$$G_w:=L^2_w(h)=\{f\in L^2(\C^n;\C^r):\|f\|^2_w:=\int_{\C^n}\|f(z)\|^2_{h_w(z)}dV(z)<\infty\}.$$
By the Griffiths subbundle formula (see \cite{B} section 2)
\be\label{eq:Gsub}
\sum_{j,k=1}^m(\Theta^G_{jk}u_j,u_k)=\Big\|\pi_\perp\sum_{j=1}^mD^G_{w_j}u_j\Big\|^2+\sum_{j,k=1}^m(\Theta^E_{jk}u_j,u_k).
\ee
where $\{u_j\}_{j=1}^m$ is an $m$-tuple of smooth sections of $E$, (i.e. $u_j$ is smooth in $w$ and holomorphic in $z$), $D$ and $\Theta$ denote the Chern connection and the curvature, and $\pi_\perp$ is the orthogonal projection on the orthogonal complement of $E$ (in $G$).

It is then straightforward to verify that $\Theta^G_{jk}=\Theta_{jk}:=\dbar_{w_k}(h^{-1}\partial_{w_j}h)$, so the Nakano positivity of $h$ implies that the left hand side of (\ref{eq:Gsub}) is positive. Hence we need to estimate the term
$$\Big\|\pi_\perp\sum_{j=1}^mD^G_{w_j}u_j\Big\|^2,$$
from above. The key observation for this is that if we set
$$v:=\pi_\perp\sum_{j=1}^mD^G_{w_j}u_j,$$
then $v$ can be regarded as a solution to the $\dbar$-equation,
\be\label{eq:dbar}
\dbar_zv=\sum_{\lambda=1}^n\sum_{j=1}^m\Theta_{j\lambda}u_jd\bar{z}_\lambda=:f,
\ee
where $\Theta_{j\lambda}=\dbar_{z_\lambda}(h^{-1}\partial_{w_j}h)$, (as the $u_j$:s are holomorphic in $z$), and $\{dz_\lambda\}_{\lambda=1}^n$ denotes an orthonormal basis for the cotangent space in the fiber direction, (orthonormal with respect to the scalar product, $(\cdot,\cdot)$, on norms induced by $h$ and the K\"ahler form $\omega$). Furthermore, since $v$ lies in the orthogonal complement of $E$, $v$ is the minimal solution to this equation. Thus, we can apply H\"ormander type $L^2$-estimates for the $\dbar$-equation to $v$.

Up until this point, we have followed the argument in \cite{B} closely. However, since the $L^2$-estimates for sections of a vector bundle and a line bundle are slightly different, the rest of the proof will be somewhat different as well.

The H\"ormander type $L^2$-estimates for (\ref{eq:dbar}) are, (see e.g. \cite{D}, Chapter VIII, Theorem 4.6, the notation used here is nonstandard and will be replaced shortly),
\be\label{eq:estimate}
\|v\|^2\leq\big(B^{-1}f,f\big)
\ee
where $B^{-1}$ is the dual of the operator,
$$B\big(\sum_{\lambda=1}^n\alpha_\lambda d\bar{z}_\lambda\big):=\sum_{\lambda,\mu=1}^n\Theta_{\lambda\mu}\alpha_\lambda d\bar{z}_\mu.$$

Combining (\ref{eq:estimate}) with (\ref{eq:Gsub}), we get that
$$\sum_{j,k=1}^m(\Theta^E_{jk}u_j,u_k)\geq\sum_{j,k=1}^m(\Theta_{jk}u_j,u_k)-\big(B^{-1}f,f\big).$$
Now by definition
$$\big(B^{-1}f,f\big)=\sup_{\alpha_1,\ldots,\alpha_n}\frac{\Big|\sum_{\lambda=1}^n(\alpha_\lambda,f_\lambda)\Big|^2}{\sum_{\lambda,\mu=1}^n(\Theta_{\lambda\mu}\alpha_\lambda,\alpha_\mu)}$$
for every $n$-tuple $\{\alpha_\lambda\}_{\lambda=1}^n$. Hence, we get that $E$ is positively curved in the sense of Nakano if for all $n$-tuples $\{\alpha_\lambda\}_{\lambda=1}^n$,
\be\label{eq:Cauchy}
\Big|\sum_{\lambda=1}^n\Big(\alpha_\lambda,\sum_{j=1}^m\Theta_{j\lambda}u_j\Big)\Big|^2\leq\Big(\sum_{j,k=1}^m(\Theta_{jk}u_j,u_k)\Big)\Big(\sum_{\lambda,\mu=1}^n(\Theta_{\lambda\mu}\alpha_\lambda,\alpha_\mu)\Big).
\ee

In order to make the proof of this inequality more transparent, we will now reformulate everything in the language of differential forms (thereby replacing $B$ with $[i\Theta,\Lambda]$). Let $\{dw_j\}_{j=1}^m$ denote an orthonormal basis for the cotangent space in the base direction.

We let $\Theta$ denote the 'total' curvature
{\setlength\arraycolsep{2pt}
\begin{eqnarray}
\Theta&=&\sum_{j,k=1}^m\Theta_{jk}dw_j\wedge d\bar{w}_k+\sum_{j=1}^m\sum_{\mu=1}^n\Theta_{j\mu}dw_j\wedge d\bar{z}_\mu+\nonumber\\
&&+\sum_{\lambda=1}^n\sum_{k=1}^m\Theta_{\lambda k}dz_\lambda\wedge d\bar{w}_k+\sum_{\lambda,\mu=1}^n\Theta_{\lambda\mu}dz_\lambda\wedge d\bar{z}_\mu.\nonumber
\end{eqnarray}}
\!\!Also, we let $\Lambda$ denote the adjoint of the operator that sends $(p,q)$-forms to $(p+1,q+1)$-forms through wedging with the K\"ahler form
$$\omega=i\sum_{j=1}^mdw_j\wedge d\bar{w}_j+i\sum_{\lambda=1}^ndz_\lambda\wedge d\bar{z}_\lambda.$$
Finally, given the $m$- and $n$-tuples $\{u_j\}_{j=1}^m$ and $\{\alpha_\lambda\}_{\lambda=1}^n$, we use them to create the vector-valued $(n+m,1)$-forms
$$u=\sum_{j=1}^mu_jdw\wedge dz\wedge d\bar{w}_j,$$
and
$$\alpha=\sum_{\lambda=1}^n\alpha_\lambda dw\wedge dz\wedge d\bar{z}_\lambda,$$
where $dw=dw_1\wedge\ldots\wedge dw_m$ and $dz=dz_1\wedge\ldots\wedge dz_n$.

A short computation now yields that
$$\Lambda u=i(-1)^{m+n}\sum_{j=1}^mu_j\widehat{dw_j}\wedge dz,$$
and
$$\Lambda\alpha=i(-1)^n\sum_{\lambda=1}^n\alpha_\lambda dw\wedge\widehat{dz_\lambda},$$
where $\widehat{dw_j}$ denotes the wedge product of all differentials $dw_k$ except $dw_j$, ordered so that $dw_j\wedge\widehat{dw_j}=dw$, and similarly for $\widehat{dz_\lambda}$. Using these, we can now calculate
$$i\Theta\wedge\Lambda u=\sum_{j,k=1}^m\Theta_{jk}u_jdw\wedge dz\wedge d\bar{w}_k+\sum_{\mu=1}^n\Big(\sum_{\mu=1}^m\Theta_{j\mu}u_j\Big)dw\wedge dz\wedge d\bar{z}_\mu,$$
and
$$i\Theta\wedge\Lambda\alpha=\sum_{k=1}^m\Big(\sum_{\lambda=1}^n\Theta_{\lambda k}\alpha_\lambda\Big)dw\wedge dz\wedge d\bar{w}_k+\sum_{\lambda,\mu=1}^n\Theta_{\lambda\mu}\alpha_\lambda dw\wedge dz\wedge d\bar{z}_\mu.$$
As $u$ and $\alpha$ are $(n+m,1)$-forms, it is immediate that $\Theta\wedge u=\Theta\wedge\alpha=0$. Hence, if we let $[i\Theta,\Lambda]$ denote the commutator between $i\Theta$ and $\Lambda$, we see that
\be\label{eq:u}
\big([i\Theta,\Lambda]u,u\big)=\sum_{j,k=1}^m(\Theta_{jk}u_j,u_k),
\ee
\be\label{eq:a}
\big([i\Theta,\Lambda]\alpha,\alpha\big)=\sum_{\lambda,\mu=1}^n(\Theta_{\lambda\mu}\alpha_\lambda,\alpha_\mu),
\ee
and
$$\big([i\Theta,\Lambda]u,\alpha\big)=\sum_{\lambda=1}^n\Big(\sum_{j=1}^m\Theta_{j\lambda}u_j,\alpha_\lambda\Big),$$
where on the left hand sides we, once again, use $(\cdot,\cdot)$ to denote the scalar product on norms induced by $h$ and the K\"ahler form $\omega$.

Altogether we see that, reformulated in terms of differential forms, (\ref{eq:Cauchy}) is equivalent to,
\be\label{eq:CauchyDiff}
\Big|\big([i\Theta,\Lambda]u,\alpha\big)\Big|^2\leq\big([i\Theta,\Lambda]u,u\big)\big([i\Theta,\Lambda]\alpha,\alpha\big).
\ee
Proving this inequality, however, is much easier than (\ref{eq:Cauchy}), since the Nakano positivity of $\Theta$ implies that $[i\Theta,\Lambda]$ is a positive operator on the space of vector-valued $(n+m,1)$-forms. (We have shown this for our special forms $u$ and $\alpha$ in (\ref{eq:u})-(\ref{eq:a}); the general case can be proven similarly, see e.g. \cite{D}, Chapter VII, Lemma 7.2.) In particular, this means that $([i\Theta,\Lambda]\cdot,\cdot)$ defines a metric on this vector space, with (\ref{eq:CauchyDiff}) just being the usual Cauchy-Schwarz inequality for this metric. Thus, (\ref{eq:Cauchy}) holds for any $n$-tuple $\{\alpha_\lambda\}_{\lambda=1}^n$ and Theorem \ref{thm:B} is proved.
\end{proof}

\begin{remark}
We have chosen to prove Theorem \ref{thm:B} for hermitian metrics on a domain $D=\Omega\times U$ in $\C^n_z\times\C^m_w$, since this setting is precisely what we need for Theorem \ref{thm:main}. The proof, however, can be adapted to the following setting.

Let $Y$ be a connected, complex manifold of dimension $m$, and let $Z$ be a compact, complex $n$-dimensional manifold. Then $X=Y\times Z$ is a complex manifold of dimension $n+m$ which can be regarded as a trivial fibration, $p:X\to Y$, over $Y$, with compact fibers $p^{-1}(w)=:X_w\cong Z$.

Let $(G,h)$ be a holomorphic, hermitian vector bundle over $X$, and let $K_Z$ denote the canonical bundle of $Z$, i.e. the bundle of forms of bidegree $(n,0)$. The Bergman spaces $A^2_w$ then get replaced by the space of global sections
$$E_w:=\Gamma\big(Z,G|_{X_w}\otimes K_Z\big),$$
where we have written $G|_{X_w}$, instead of $G|_Z$, to stress that the metric $h$ on $G$ depends on the base point $w$.

Now if $X$ is K\"ahler, then just as in the setting Theorem 1.2 in \cite{B}, $E$ has a natural structure as a holomorphic vector bundle. Moreover, as elements of $E_w$ can be integrated over the fiber $X_w=Z$, (with respect to the metric $h_w$ and the K\"ahler form on $X$), we also have a natural, nontrivial, metric $\|\cdot\|_w$ on $E$.

Adapted to this setting, the proof of Theorem \ref{thm:B} yields the following theorem.
\begin{thm}
If $X$ is K\"ahler and $(G,h)$ is positively curved in the sense of Nakano over $X$, then $(E,\|\cdot\|_w)$ is also positively curved in the sense of Nakano.
\end{thm}
A much more general version of this theorem has previously been obtained by Mourougane and Takayama, \cite{MT} Theorem 1.1, (see also \cite{LY}).
\end{remark}

\section{The proof of Theorem \ref{thm:main}}
\noindent With Theorems \ref{thm:PW} and \ref{thm:B} at our disposal, we can finally turn to the proof of our matrix-valued Pr\'ekopa theorem.

\begin{proof}[Proof of Theorem \ref{thm:main}] Given the metric $g:\R^n_y\times\R^m_t\to\C^{r\times r}$, we can think of it as a metric $h:\C^n_z\times\C^m_w\to\C^{r\times r}$ which is independent of the real parts of $z=x+iy$ and $w=s+it$, i.e. $h(z,w)=g(y,t)$. Then $h$ will be positively curved in the sense of Nakano, since $g$ is log concave in the sense of Nakano by assumption. Thus we are in the setting of Theorem \ref{thm:B} and so we know that the trivial, (infinite rank) holomorphic vector bundle $(E,\|\cdot\|_{w,h})$ over $\C^m$, with fiber $E_w=A^2_w(h)$ as in (\ref{eq:fiber}), is positively curved in the sense of Nakano as well.

As sketched at the end of the introduction, we now use $h$ to define the metric $\tilde{h}:\R^n_\xi\times\C^m_w\to\C^{r\times r}$ through
\be\label{eq:herm_pre1}
\tilde{h}(\xi,w)=(2\pi)^n\int_{\R^n}e^{2\xi\cdot y}h(y,w)dV(y).
\ee
Since $h(z,w)$ is hermitian and independent of the real parts of $z$ and $w$, the same thing holds for $\tilde{h}$. In fact, comparing with (\ref{eq:vb_pre}) we see that $\tilde{h}(\xi,w)=\tilde{g}(\xi,t)$.

From (\ref{eq:herm_pre1}) it follows that for each fix $w\in\C^m$, $\tilde{h}_w(\cdot):=\tilde{h}(\cdot,w)$ is a hermitian metric on $\R^n$, and so in this way we can construct a second trivial, (infinite rank), holomorphic vector bundle $(\tilde{E},\|\cdot\|_{w,\tilde{h}})$ over $\C^m$, with fiber $\tilde{E}_w:=L^2_w(\tilde{h})$ as in (\ref{eq:L2}).

Theorem \ref{thm:PW}, and (\ref{eq:parseval}) in particular, then yield that the vector bundles $(E,\|\cdot\|_{w,h})$ and $(\tilde{E},\|\cdot\|_{w,\tilde{h}})$ are isometrically isomorphic. Hence, since $(E,\|\cdot\|_{w,h})$ is positively curved in the sense of Nakano, so is $(\tilde{E},\|\cdot\|_{w,\tilde{h}})$.

Let $f$ and $k$ denote two sections of $\tilde{E}$, so that for each fix $w\in\C^m$, $f_w,k_w\in L^2_w(\tilde{h})$. As we are investigating a pointwise property, (namely being curved in the sense of Nakano), we can without any loss of generality assume that $f$ and $k$ both are holomorphic in $w$. The Chern connection of $(\tilde{E},\|\cdot\|_{w,\tilde{h}})$ is given by $D=\nabla+\dbar$, where $\nabla$ is defined through $(\nabla f,k)_{w,\tilde{h}}:=\partial(f,k)_{w,\tilde{h}}$. Hence a short computation yields that
{\setlength\arraycolsep{2pt}
\begin{eqnarray}
(\nabla f,k)_{w,\tilde{h}}&=&\partial_w\int_{\R^n}\big(f_w(\xi),k_w(\xi)\big)_{\tilde{h}_w(\xi)}dV(\xi)=\nonumber\\
&&\qquad\qquad\qquad=\int_{\R^n}\big(D_{w,\xi}'f_w(\xi),k_w(\xi)\big)_{\tilde{h}_w(\xi)}dV(\xi),\nonumber
\end{eqnarray}}
\!\!where $D_{w,\xi}':=\partial_w+\tilde{h}^{-1}(\xi,w)\partial_w\tilde{h}(\xi,w)$.

For the curvature of $(\tilde{E},\|\cdot\|_{w,\tilde{h}})$, an equally short computation gives that
{\setlength\arraycolsep{2pt}
\begin{eqnarray}
&&(\Theta^{\tilde{E}}f,k)_{w,\tilde{h}}:=(D^2f,k)_{w,\tilde{h}}=\big((\nabla\dbar+\dbar\nabla)f,k\big)_{w,\tilde{h}}=(\dbar\nabla f,k)_{w,\tilde{h}}=\nonumber\\
&&\quad=\int_{\R^n}\big(\dbar_wD_{w,\xi}'f_w(\xi),k_w(\xi)\big)_{\tilde{h}}dV(\xi)=\int_{\R^n}\big(\Theta^{\tilde{h}}_{\xi,w}f_w(\xi),k_w(\xi)\big)_{\tilde{h}}dV(\xi),\nonumber
\end{eqnarray}}
\!\!where $\Theta^{\tilde{h}}_{\xi,w}:=\Theta^{\tilde{h}}(\xi,w)=\dbar_w\big(\tilde{h}^{-1}(\xi,w)\partial_w\tilde{h}(\xi,w)\big)$. In particular, this shows that
$$(\Theta^{\tilde{E}}_{jl}f,k)_{w,\tilde{h}}=\int_{\R^n}\big(\Theta^{\tilde{h}}_{jl}(\xi,w)f_w(\xi),k_w(\xi)\big)_{\tilde{h}}dV(\xi),$$
for all $j,l=1,\ldots,m$.

By definition, $\tilde{E}$ being positively curved in the sense of Nakano means that for any $m$-tuple $\{f^j\}_{j=1}^m$ of sections of $\tilde{E}$,
$$\sum_{j,l=1}^m\big(\Theta^{\tilde{E}}_{jl}f^j,f^l\big)_{w,\tilde{h}}\geq0.$$
Thus, from what we have just seen this implies that
\be\label{eq:h_Nak}
\int_{\R^n}\sum_{j,l=1}^m\big(\Theta^{\tilde{h}}_{jl}(\xi,w)f^j_w(\xi),f^l_w(\xi)\big)_{\tilde{h}_w(\xi)}dV(\xi)\geq0.
\ee

Now as $\{f^j\}_{j=1}^m$ are arbitrary sections of $\tilde{E}$, we can multiply each $f^j$ with a cut-off function $\chi(\xi)$, since if $f^j_w(\xi)\in L^2_w(\tilde{h})$ for each $w$, so is $\chi(\xi)f^j_w(\xi)$. But this means that we without any loss of generality can replace (\ref{eq:h_Nak}) with
$$\int_{\R^n}\sum_{j,l=1}^m\big(\Theta^{\tilde{h}}_{jl}(\xi,w)f^j_w(\xi),f^l_w(\xi)\big)_{\tilde{h}_w(\xi)}\chi^2(\xi)dV(\xi)\geq0.$$
As $\chi$ is arbitrary, we get that
$$\sum_{j,l=1}^m\big(\Theta^{\tilde{h}}_{jl}(\xi,w)f^j_w(\xi),f^l_w(\xi)\big)_{\tilde{h}_w(\xi)}\geq0,$$
for all $\xi\in\R^n$.

What we have shown is that
$$\tilde{h}(\xi,w)=(2\pi)^n\int_{\R^n}e^{2\xi\cdot y}h(y,w)dV(y)$$
is Nakano positive in $w$. In particular, since we have that $\tilde{h}(\xi,w)=\tilde{g}(\xi,t)$ and $h(z,w)=g(y,t)$, this means that for all $\xi\in\R^n$,
$$\tilde{g}(\xi,t)=(2\pi)^n\int_{\R^n}e^{2\xi\cdot y}g(y,t)dV(y),$$
is log concave in the sense of Nakano in $t$. Thus choosing $\xi=0$ finishes the proof of the theorem.
\end{proof}

\begin{remark}
Just as for the original Pr\'ekopa theorem, it is straightforward to extend Theorem \ref{thm:main} to integration over arbitrary convex sets $K$ in $\R^n$. Namely, let $\phi_K:K\to\R\cup\{\infty\}$ be the function
$$\phi_K(x):=\left\{ \begin{array}{ll}
0 & \textrm{if $x\in K$,}\\
\infty & \textrm{otherwise.}
\end{array} \right.$$
Then, $\phi_K$ can be written as the limit of an increasing sequence of convex functions $\{\phi_j\}_{j=1}^\infty$.

If  $g:\R^n_y\times\R^m_t\to\C^{r\times r}$ is a metric that is log concave in the sense of Nakano, then by arguing as in example \ref{ex:2} and \ref{ex:3} of section 2, we have that for each $j$, the metric $e^{-\phi_j}g$ is also Nakano log concave. Hence, by Theorem \ref{thm:main},
$$\int_{\R^n}e^{-\phi_j(y,t)}g(y,t)dV(y)$$
is log concave in the sense of Nakano.

Letting $j\to\infty$ and using monotone convergence, we thus get that
$$\tilde{g}(t):=\int_Kg(y,t)dV(y),$$
is log concave in the sense of Nakano as well.
\end{remark}

\end{document}